\documentclass[10pt]{amsart}
\usepackage{amsmath, amsfonts, amssymb, amstext, amscd, amsthm}
\numberwithin{equation}{section}
\allowdisplaybreaks

\newtheorem{art_conj}{Conjecture}[section]
\newtheorem{dirichlet}[art_conj]{Theorem}
\newtheorem{clss-fld}[art_conj]{Theorem}
\newtheorem{dirichlet-cor}[art_conj]{Corollary}
\newtheorem{art_prop}[art_conj]{Proposition}
\newtheorem{langlands}[art_conj]{Theorem}
\newtheorem{lang-cor}[art_conj]{Corollary}
\newtheorem{klein}[art_conj]{Theorem}
\newtheorem{l-series}[art_conj]{Lemma}
\newtheorem{q-curve}[art_conj]{Proposition}
\newtheorem{modularity}[art_conj]{Proposition}
\newtheorem{deformations}[art_conj]{Theorem}

\begin{document}


\title{Artin's Conjecture and Elliptic Curves}

\author{Edray Herber Goins}

\address{Institute for Advanced Study \\ School of Mathematics \\ Olden Lane \\ Princeton, NJ 08540}


\email{goins@math.ias.edu}

\thanks{This research was sponsored by a joint fellowship from the National
Physical Science Consortium (NSPC) and the National Security Agency (NSA),
as well as a grant from the NSF, number DMS 97-29992.}


\begin{abstract}
Artin conjectured that certain Galois representations should give rise
to entire L-series.  We give some history on the conjecture and
motivation of why it should be true by discussing the one-dimensional
case.  The first known example to verify the conjecture in
the icosahedral case did not surface until Buhler's work in 1977.  We
explain how this icosahedral representation is attached to a modular
elliptic curve isogenous to its Galois conjugates, and then explain how
it is associated to a cusp form of weight 5 with level prime to 5.
\end{abstract}

\maketitle


\section{Introduction}

In 1917, Erich Hecke \cite{MR89k:01043} proved a series of results
about certain characters which are now commonly referred to as Hecke
characters;
one corollary states that one-dimensional complex Galois representations give
rise to entire L-series.  He revealed, through a series of lectures
\cite{MR85c:11042} at Princeton's Institute for Advanced Study in the years
that followed, the relationship between such representations as
generalizations of Dirichlet characters and modular forms as the eigenfunctions of a set
of commuting self-adjoint operators.
Many mathematicians were inspired by his ground-breaking insight
and novel proof of the analytic continuation of the L-series.

In the 1930's, Emil Artin \cite{MR31:1159} conjectured
that a generalization of such a result should be true; that is,
irreducible complex projective representations of finite Galois
groups should also give rise to entire L-series.  He came
to this conclusion after proving himself that both
3-dimensional and 4-dimensional representations of the simple group of
order 60, the
alternating group on five letters, might give rise to L-series with
singularities.  It
is known, due to the insight of Robert Langlands \cite{MR81e:12013} in
the 1970's relating Hecke characters with Representation Theory, that
in order to prove the conjecture it suffices
to prove that such representations are associated to cusp forms.  This conjecture
has been the motivation for much study in both Algebraic and Analytic
Number Theory ever since.

In this paper, we present an elementary approach to Artin's
Conjecture by considering the problem over $\mathbb Q$.  We consider
Dirichlet's theorem which preceeded Hecke's results, and
sketch a proof by introducing theta series.  We then introduce
Langland's program to exhibit cusp forms.  We
conclude by studying a specific example which is associated to an
elliptic curve.  We assume in the final sections that the reader is
somewhat familiar with
the basic properties of elliptic curves.


\section{One-Dimensional Representations}

We begin with some classical definitions and theorems.  We are
motivated by the expositions in \cite{MR92e:11001} and \cite{MR97k:11080}.

\subsection{Reimann Zeta Function and Dirichlet L-Series}

Let $N$ be a fixed positive integer, and $\chi: \left( \mathbb Z / N \mathbb Z
\right)^{\times} \to \mathbb C^{\times}$ be a group homomorphism.  We
extend $\chi: \mathbb Z \to \mathbb C$ to the entire ring of
integers by defining 1) $\chi(n) = \chi(n \mod N)$ on the
residue class modulo $N$; and 2) $\chi(n) = 0$ if $n$ and $N$ have a factor in
common.  One easily checks that this extended definition still yields a
multiplicative map
i.e. $\chi(n_1 \, n_2) = \chi(n_1) \, \chi(n_2)$.

Fix a complex number $s \in \mathbb C$ and associate the {\em L-series}
to $\chi$ as the sum

\begin{equation}
    L(\chi, \, s) = \chi(1) + \frac {\chi(2)}{2^s} + \frac
    {\chi(3)}{3^s} + \dots = \sum_{n=1}^{\infty} \frac {\chi(n)}{n^s}
\end{equation}

\noindent Quite naturaly, two questions arise:

{\em \begin{center}
For which region is this series a well-defined function?

Can that function be continued analytically to the entire complex plane?
\end{center}}

\noindent This series is reminiscent of the {\em Reimann zeta
function}, defined as

\begin{equation}
    \zeta(s) = 1 + \frac 1{2^s} + \frac 1{3^s} + \dots =
\sum_{n=1}^{\infty} \frac 1{n^s}
\end{equation}

\noindent We may express this sum in a slightly different fashion.  It
is easy to check that\footnote{In the literature, it is standard to define
\[ \Gamma(s) = \int_0^{\infty} e^{-y} \, y^{s-1} \, dy = (s-1)!,
\qquad \text{Re}(s) \geq 1\]
as the Generalized Factorial or Gamma function.}

\begin{equation}
   \int_0^{\infty} e^{-n y} \, y^{s-1} \, dy = \frac {(s-1)!}{n^s}
\end{equation}

\noindent so that the sum becomes

\begin{equation}
    \zeta(s) = \sum_{n=1}^{\infty} \frac 1{n^s} = \sum_{n=1}^{\infty}
    \frac 1{(s-1)!} \, \int_0^{\infty} e^{-n y} \, y^{s-1} \, dy =
    \frac 1{(s-1)!} \, \int_0^{\infty} \frac {y^{s-1}}{e^y - 1} \, dy
\end{equation}

\noindent The function $e^y$ grows faster that any power of $y$, so
the integrand converges for all complex $s-1$ with positive real part.
However, we must be careful when $y=0$; the denominator of the
integral vanishes, so we need the numerator to vanish as well.  This
is not the case if $s = 1$.  Hence, we find the

\begin{dirichlet}[Dirichlet \cite{MR92e:11001}] \label{dirichlet}
$L(\chi, \, s)$ is a convergent series if $\text{Re}(s) > 1$.
If $\chi$ is not the trivial character $\chi_0 = 1$ (i.e. $N \neq 1$) then
$L(\chi, \, s)$ has analytic continuation everywhere.  On the other hand,
$L(\chi_0, \, s) = \zeta(s)$ has a pole at $s=1$.  In either case,
the L-series has the product expansion

\begin{equation}
    L(\chi, \, s) = \prod_{\text{primes } p} \left( 1 - \frac{\chi(p)}{p^s}
\right)^{-1}
\end{equation}

\end{dirichlet}

\begin{proof} We present a sketch of proof when $N=5$, although the
method works in general.  Consider the character

\begin{equation}
    \left( \frac 5n \right) = \begin{cases} +1 & \text{if $n \equiv
    1, \, 4 \mod 5$;} \\ -1 & \text{if $n \equiv 2, \, 3 \mod 5$;} \\
    \ \ 0 & \text{if $n$ is divisible by 5.} \end{cases}
\end{equation}

\noindent The L-series in this case may be expressed as the integral

\begin{equation}
    L \left( \left( \frac 5* \right), \, s \right) = \frac 1{(s-1)!}
    \, \int_0^{\infty} \left[ \sum_{n=1}^{\infty} \left( \frac 5n
    \right) \, e^{-n y} \right] \, y^{s-1} \, dy
\end{equation}

\noindent The infinite sum inside the integrand can be evaulated by
considering different cases:  When $n$ is divisible by 5, the character
vanishes.  Otherwise write $n = 5 \, k + 1, \ldots, 5 \, k + 4$ in the
other cases so that the sum becomes

\begin{equation} \begin{aligned}
    \sum_{n=1}^{\infty} \left( \frac 5n \right) \, e^{-n y} & =
    \sum_{k=0}^{\infty} \left[ + e^{-(5k+1)y} - e^{-(5k+2)y} -
    e^{-(5k+3)y} + e^{-(5k+4)y} \right] \\
    & = \frac {e^{-y} - e^{-2y} - e^{-3y} + e^{-4y}}{1 - e^{-5y}}
    = \frac {e^{3y} - e^y}{1 + e^y + e^{2y} + e^{3y} + e^{4y}}
\end{aligned} \end{equation}

\noindent which gives the L-series as

\begin{equation}
    L \left( \left( \frac 5* \right), \, s \right) = \frac 1{(s-1)!}
    \, \int_0^{\infty} \frac {e^{3y} - e^y}{1 + e^y + e^{2y} +
    e^{3y} + e^{4y}} \ y^{s-1} \ dy
\end{equation}

\noindent Hence the integrand is analytic at $y=0$, so the L-series
does not have a pole at $s=1$. \end{proof}

\subsection{Galois Representations}

Fix $q(x)$ as an irreducible polynomial of degree $d$ with leading
coefficient 1 and
integer coefficients, and set $K/\mathbb Q$
as its splitting field.  The group of permutations of the roots
$\text{Gal}(K/\mathbb Q)$ has a canonical representation as $d \times
d$ matrices.  To see why, write the $d$ roots $q_k$ of $q(x)$ as
$d$-dimensional unit vectors:

\begin{equation}
    q_1 = \begin{pmatrix} 1 \\ 0 \\ \vdots \\ 0 \end{pmatrix}; \qquad
    q_2 = \begin{pmatrix} 0 \\ 1 \\ \vdots \\ 0 \end{pmatrix}; \qquad
    \ldots; \qquad
    q_d = \begin{pmatrix} 0 \\ 0 \\ \vdots \\ 1 \end{pmatrix}. \qquad
\end{equation}

\noindent  Any permutation $\sigma$ on these roots may be
representated as a $d \times d$ matrix $\varrho(\sigma)$.  One easily
checks that $\varrho$ is a multiplicative map i.e. $\varrho(\sigma_1 \,
\sigma_2) = \varrho(\sigma_1) \, \varrho(\sigma_2)$.
As an example, consider $x^2 + x - 1$.  Then $K = \mathbb
Q(\sqrt{5})$, and the only permutation of interest is

\begin{equation}
    \sigma: \quad \frac {-1 + \sqrt{5}}2 \mapsto \frac {-1 - \sqrt{5}}2
    \qquad \implies \qquad
    \varrho(\sigma) = \begin{pmatrix} & 1 \\ 1 \end{pmatrix}.
\end{equation}

Denote $G_{\mathbb Q}$, the {\em absolute Galois group}, as the union of
each of the $\text{Gal}(K/\mathbb Q)$ for such polynomials $q(x)$ with
integer coefficients.  This larger group still
permutes the roots of a specific polynomial $q(x)$; we consider it
because it is a universal object independent of $q(x)$.  We view the
permutation representation as a group homomorphism
$\varrho: G_{\mathbb Q} \to GL_d(\mathbb C)$.

The finite collection of matrices $\{ \varrho(\sigma) | \sigma \in
G_{\mathbb Q}\}$ acts on the $d$-dimensional
complex vector space $\mathbb C^d$, so we are concerned with {\em lines}
which are invariant under the action of all of the $\varrho(\sigma)$.  That
is, if we attempt to simultantously diagonalize all of the $\varrho(\sigma)$
then we want to consider one-dimensional invariant subspaces.
For example, for $x^2 + x - 1$ we may instead choose the basis

\begin{equation}
    {q_1}' = \frac 1{\sqrt{2}} \, \begin{pmatrix} 1 \\ 1 \end{pmatrix} \, \quad
    {q_2}' = \frac 1{\sqrt{2}} \, \begin{pmatrix} \ \ 1 \\ -1
    \end{pmatrix} \qquad \implies \qquad \varrho'(\sigma) = \begin{pmatrix}
    1 \\ & -1 \end{pmatrix};
\end{equation}

\noindent so that we have the more intuitive representation
$\rho(\sigma) = -1$ defined on the eigenvalues.  In general, we do not
wish to consider the $d \times d$ matrix representation $\varrho$, but
rather the scalar $\rho: G_{\mathbb Q} \to \mathbb C^{\times}$.

In order to define an L-series, we use the product expansion found
in \ref{dirichlet} above.  To this end, choose a prime number $p$ and
factor $q(x)$ modulo $p$, say in the form

\begin{equation}
    q(x) \equiv \left( x^{f_1} + \dots \right)^{e_1} \, \left( x^{f_2} +
    \dots \right)^{e_2} \hdots \left( x^{f_r} + \dots \right)^{e_r} \mod p
\end{equation}

\noindent If each $e_j = 1$, we say $p$ is {\em unramified} (and
{\em ramified} otherwise).  In this case, there is a universal automorphism
$\text{Frob}_p \in G_{\mathbb Q}$ which yields surjective maps
$G_{\mathbb Q} \to \mathbb Z / f_j \, \mathbb Z$ for $j = 1, \ldots ,
r$.  This automorphism is canonically defined by the congruence

\begin{equation}
\text{Frob}_p \left( q_k \right) \equiv {q_k}^p \mod p \qquad
\text{on the roots $q_k$ of $q(x)$.}
\end{equation}

\noindent Unfortunately, when $p$ is ramified there is not a canonical
choice of
Frobenius element because there are repeated roots.  We will denote $\Sigma$
as a finite set containing the ramified primes.

If $\rho$ is a map $G_{\mathbb Q} \to \mathbb C^{\times}$, the Frobenius
element
induces a map $G_{\mathbb Q} \to \mathbb Z / f_j \, \mathbb Z$, and Dirichlet
characters are maps $\mathbb Z / f_j \, \mathbb Z \to \mathbb C$, it
should follow that Dirichlet characters $\chi$ are closely related to such
maps $\rho$.  This is indeed the case.

\begin{clss-fld}[Artin Reciprocity \cite{MR92e:11001}]
Fix $q(x)$ and $\rho: G_{\mathbb Q} \to \mathbb C^{\times}$ be as
described above.  Define $\chi_{\rho}: \mathbb Z \to \mathbb C$ on primes
by the identification

\begin{equation}
    \chi_{\rho}(p) =
    \begin{cases} \rho(\text{Frob}_p) & \text{when unramified,} \\
    0 & \text{when ramified;} \end{cases}
\end{equation}

\noindent and extend $\chi_{\rho}$ to all of $\mathbb Z$ by multiplication.
Then there exists a positive integer $N = N(\rho)$, called the
conductor of $\rho$, divisible only by the primes which ramify such that
$\chi_{\rho}$ is a Dirichlet character modulo $N$.  \end{clss-fld}

As an example, $x^2 + x - 1$ factors modulo the first few primes as

\begin{equation} \begin{matrix}
    x^2 + x - 1 & \equiv
    & x^2 + x + 1 & \mod 2 \\
    & \equiv & x^2 + x + 2 & \mod 3 \\
    & \equiv & (x+3)^2     & \mod 5 \\
    & \equiv & x^2 + x + 6 & \mod 7 \\
    & \equiv & (x+4)(x+8)  & \mod 11 \\
\end{matrix} \end{equation}

\noindent Then $\text{Frob}_2$, $\text{Frob}_3$, and $\text{Frob}_7$
are each nontrivial automorphisms, while $\text{Frob}_{11}$ is
the identity.  The only ramified prime is 5 because of the repeated roots.
We map

\begin{equation}
    \rho(\text{Frob}_2) = \rho(\text{Frob}_3) = \rho(\text{Frob}_7) =
    -1; \qquad \rho(\text{Frob}_{11}) = +1.
\end{equation}

\noindent The asociated Dirichlet character $\chi_{\rho} = \left( \frac 5*
\right)$
is just the character modulo $N=5$ we considered above.

\subsection{Artin L-Series: 1-Dimensional Case}

We are now in a position to define and study the L-series associated to
Galois representations.

\begin{dirichlet-cor}[Artin, Dirichet] \nocite{MR95f:11085}
Given a map $\rho: G_{\mathbb Q} \to \mathbb C^{\times}$ as defined above with
$\Sigma$ a finite set containing the ramified primes, define the Artin
L-series as

\begin{equation}
    L_{\Sigma}(\rho, \, s) = \sum_{n=1}^{\infty} \frac {\chi_{\rho}(n)}{n^s}
    = \prod_{p \notin \Sigma} \left( 1 - \frac {\rho(\text{Frob}_p)}{p^s}
\right)^{-1}
\end{equation}

Then $L(\rho, \, s)$ converges if $\text{Re}(s) > 1$.  If $\rho$ is
not the trivial map $\rho_0 = 1$ then $L_{\Sigma}(\rho, \, s)$ has
analytic continuation everywhere.  On the other hand,

\begin{equation}
    L_{\Sigma}(\rho_0, \, s) = \zeta(s) \cdot \prod_{p \in \Sigma} \left(1
- p^s \right)
\end{equation}

\noindent has a pole at $s=1$.  \end{dirichlet-cor}

\begin{proof} We sketch the proof.  First, we express the L-series in terms
of the integral of
a function which dies exponentially fast.  This will guarantee that the
L-series
converges for some right-half plane.  For $\tau = x+i \, y$ in the
upper-half plane (i.e. $y > 0$) define

\begin{equation}
    \theta_{\rho}(\tau) = \sum_{n=1}^{\infty} \chi_{\rho}(n) \,
    n^{\epsilon} \, e^{\pi i \, n^2 \, \tau} \qquad \text{where} \qquad
    \epsilon = \begin{cases} 0 & \text{if $\chi_{\rho}(-1) = +1$,} \\
    1 & \text{if $\chi_{\rho}(-1) = -1$;} \\
\end{cases} \end{equation}

\noindent so that the L-series may be expressed as the integral

\begin{equation}
    L(\rho, \, s) = \frac {\pi^{\frac {s+\epsilon}2}}{\left( \frac
    {s+\epsilon-2}2 \right)!}
    \, \int_0^{\infty} \theta_{\rho}(iy) \, y^{\frac {s+\epsilon-2}2} \, dy
\end{equation}

The integrand may have a pole at $y=0$, so we perform our second
trick.  Break the integral up into two regions $0 < y < 1$ and $1 <
y$, and then use the functional equation

\begin{equation}
    \theta_{\overline \rho} \left( - \frac 1{N^2 \, \tau} \right) =
    w(\rho) \, N^{\epsilon + \frac 12} \, \tau^{\epsilon + \frac 12} \
\theta_{\rho}(\tau)
\end{equation}

\noindent --- where $w(\rho)$ is a complex number of absolute value 1, $N =
N(\rho)$ is the conductor, and $\overline \rho$ is the complex
conjugate --- to express the integral solely in terms of values $1/N < y$:

\begin{equation}
    \int_0^{\infty} \theta_{\rho}(iy) \, y^{\frac {s+\epsilon-2}2}
    \, dy = \int_{1/N}^{\infty} \left[ \theta_{\rho}(iy) \, y^{\frac
    {s+\epsilon-2}2} + w(\overline \rho) \, N^{-s+ \frac 12} \,
    \theta_{\overline \rho}(iy) \, y^{\frac {-s+\epsilon-1}2} \right] dy
\end{equation}

\noindent Hence the integral converges for all complex $s$.  \end{proof}


\section{Artin's Conjecture}

\subsection{Artin L-Series: General Case}

Fix $q(x)$ as an irreducible polynomial of degree $d$ with leading
coefficient 1 and integer coefficients, let $\Sigma$ be a finite set
containing the primes which ramify, and set $G_{\mathbb Q}$ as the
absolute Galois group as before.  We found that there is a canonical
permutation $\varrho: G_{\mathbb Q} \to GL_d(\mathbb C)$ induced by
the action on the roots.  Consider the ring

\begin{equation}
    V_{\varrho} = \left \{ \left. \sum_{\sigma \in G_{\mathbb Q}}
\lambda_{\sigma}
    \, \varrho(\sigma) \in GL_d(\mathbb C) \right| \lambda_{\sigma}
    \in \mathbb C \right \}
\end{equation}

\noindent generated by the linear combinations of matrices in the
image of $\varrho$.  We view this as a complex vector space, which is acted
upon
by the linear transformations $\varrho(\sigma)$ quite naturally by matrix
multiplication.  As with any complex vector space, we may decompose it
into invariant subspaces.  Before, we considered only one-dimensional spaces,
but now we generalize to an arbitrary invariant irreducible subspace $V \subseteq
V_{\varrho}$.
We restrict $\varrho$ such that the action is faithful on this
subspace.  That is,

\begin{equation}
    \rho = \varrho |_V: \quad G_{\mathbb Q} \to GL(V) \qquad \text{is
    irreducible.}
\end{equation}

Define the L-series associated to $\rho$ as the product

\begin{equation}
    L_{\Sigma}(\rho, \, s) = \prod_{p \notin \Sigma}
    \det \left( 1 - \frac {\rho(\text{Frob}_p)}{p^s} \right)^{-1}
    = \sum_{n=1}^{\infty} \frac {a_{\rho}(n)}{n^s}
\end{equation}

\noindent As before, there is a corresponding $N = N(\rho)$, called the {\em
conductor}, which is divisible by the primes $p \in \Sigma$.  By
considering the determinant, we find a Dirichlet character
$\epsilon_{\rho}$, called the {\em nebentype}, which is
associated to the one-dimensional Galois representation $\det \circ
\rho$.  The coefficients $a_{\rho}(n)$ are closely related to the
Frobenius element:

\begin{equation}
    a_{\rho}(p) =
    \begin{cases} \text{tr} \, \rho(\text{Frob}_p) & \text{when unramified,} \\
    0 & \text{when ramified.} \end{cases}
\end{equation}

\noindent Quite naturally, two questions arise once again:

{\em \begin{center}
    For which complex numbers $s$ is this series a well-defined function?

    Can that function be continued analytically to the entire complex plane?
\end{center}}

\noindent The first question has an answer which is consistent with
the theme so far.

\begin{art_prop}[Artin \cite{MR31:1159}]
$L_{\Sigma}(\rho, \, s)$ converges if $\text{Re}(s) > 1$.  The L-series
associated to the trivial map $L_{\Sigma}(\rho_0, \, s)$ has a pole at
$s=1$.  \end{art_prop}

\noindent However, the second question remains an open problem.

\begin{art_conj}[Artin] If $\rho$ is irreducible, not trivial, and is
unramified outside a finite set of primes $\Sigma$, then
$L_{\Sigma}(\rho, \, s)$ has analytic continuation everywhere.
\end{art_conj}

The case of one-dimensional Galois representations (i.e. $V \simeq
\mathbb C$) was proved in full
generality with the advent of Class Field Theory.  Indeed, any
one-dimensional Galois representation must necessarily be abelian, so that
by Artin Reciprocity the representation can be associated with a character
defined on the idele group.   When working over $\mathbb Q$, this
amounts to saying that every one-dimensional representation may be associated
with a Dirichlet character.

Many mathematicians, inspired by this result, began work on the
irreducible two-dimensional representations (i.e. $V \simeq \mathbb
C^2$).  Felix Klein \cite{MR96g:01046}
had showed that the only finite projective images in the complex
general linear group correspond to
the Platonic Solids; that is, they are the rotations of the regular polygons
and regular polyhedra.  This is because we have the injective map

\begin{equation}
    SO_3(\mathbb R) = \left \{ \left. \gamma \in \text{Mat}_3(\mathbb R)
\right| \det
    \gamma = 1, \ {\gamma}^t = {\gamma}^{-1} \right \} \to PGL_2(\mathbb C)
\end{equation}

\noindent which relates rotations of three-dimensional symmetric objects
with $2 \times 2$ matrices modulo scalars.  Hence, it suffices to consider
irreducible
two-dimensional projective representations with these images
in order to prove Artin's Conjecture in this case.

Most of these cases of the conjecture have been answered in the affirmative.
Irreducible cyclic and dihedral representations (that is,
representations whose image in $PGL_2(\mathbb C)$ is isomorphic to
$Z_n$ or $D_n$, respectively) may be interpreted as representations
induced from abelian ones, so that the proof of analytic continuation may be
reduced to one using Dirichlet characters.  Irreducible
tetrahedral and some octahedral representations (i.e. projective image
isomorphic to
$A_4$ or $S_4$, respectively) were proved to give entire
L-series due to work by Robert Langlands \cite{MR82a:10032} in the
1970's on base change for $GL(2)$.  The remaining cases for irreducible
octahedral representations were proved shortly thereafter by Jerrold Tunnell
\cite{MR82j:12015}.  Such methods worked because they exploited the existence
proper nontrivial normal subgroups.  Unfortunately, the
simple group of order 60 has none, so it is still not known whether
the irreducible icosahedral representations (i.e. projective image
isomorphic to $A_5$) have analytic continuation.  The first known
example to verify Artin's conjecture in this case did not surface
until Joe Buhler's work \cite{MR58:22019} in 1977.  It is this
example we wish to consider in detail.

For general finite dimensional representations $\rho: G_{\mathbb Q}
\to GL(V)$, not much is known.
It is easy to show that the L-series is analytic in a right-half of the
complex plane.
In 1947, Richard Brauer \cite{MR8:503g} proved that the characters associated
to representations of finite groups are a
finite linear combination of one-dimensional characters, and so the
corresponding L-series have {\em meromorphic} continuation; that is,
the functions have at worst poles at a finite number of places.
Brauer's proof does not guarantee that the integral
coefficients of such a linear combination are positive; it
can be shown that in many cases the coefficients are negative so that
the proof of continuation to the entire complex plane may be reduced
to showing that the poles of the L-series are cancelled by the zeroes.

\subsection{Maass Forms and the Langlands Program}

Robert Langlands completed a circle of ideas which related L-series,
complex representations, and automorphic representations.  While the
deep significance of these ideas is far beyond the scope of this
paper, we will content ourselves with a simplified consequence.

\begin{langlands}[Langlands \cite{MR82a:10032}] \label{langlands}
In order to prove Artin's Conjecture for two-dimensional
representations $\rho$ unramified outside a finite set of primes, it suffices
to prove that

\begin{equation} \label{maass}
    f_{\rho}(\tau) = \sum_{n=0}^{\infty} a_{\rho}(n) \, y^{\frac k2} \,
    e^{2 \pi i \, n \, \tau} \qquad (\tau = x + iy)
\end{equation}

\noindent is a Maass cusp form of weight $k=1$, level $N = N(\rho)$,
and nebentype $\epsilon_{\rho} = \det \circ \rho$.
\end{langlands}

The idea is to express the L-series as the integral of a function
which dies exponentially fast.  Indeed, we have the relation

\begin{equation}
    L_{\Sigma}(\rho, \, s) = \frac {2^s \, \pi^s}{(s-1)!} \,
    \int_0^{\infty} f_{\rho}(iy) \, y^{s- \frac k2 - 1} \, dy
\end{equation}

\noindent and, by definition, the function dies off exponentially fast
as $y \to \infty$.

In general, a smooth function $f: \{ x+iy \, | \, y>0 \} \to \mathbb C$ is
called a {\em Maass form} of weight $k \in \mathbb Z$, level $N \in \mathbb
Z$, and nebentype $\varepsilon : \mathbb Z \to \mathbb C$
if it satisfies the following properties:

\begin{enumerate}
    \item {\em Eigenfunction of the non-Euclidean Laplacian.}
    $f(x+iy)$ satisfies the differential equation
    \begin{equation}
	\left \{ -y^2 \left( \frac {\partial^2}{\partial x^2} +
	\frac {\partial^2}{\partial y^2} \right)
	+ i \, k \, y \, \frac {\partial}{\partial x} \right \} \, f(\tau)
	= \frac k2 \left( 1-\frac k2 \right) \, f(\tau)
    \end{equation}
    \item {\em Exponential Decay.}  For any complex $s$ with
    $\text{Re}(s) > 1$,
    \begin{equation}
	\lim_{y \to \infty} f(x + i y)  \, y^{s-1}= 0
    \end{equation}
    \item {\em Transformation Property.} For all matrices in the group

    \begin{equation}
	\Gamma_0(N) = \left \{ \left. \begin{pmatrix} a & b \\ c & d
\end{pmatrix}
	\in \text{Mat}_2(\mathbb Z) \right|
	\text{$a d - b c = 1$ and $c$ is divisible by $N$} \right \}
    \end{equation}
    \noindent we have the identity
    \begin{equation}
	f \left( \frac {a \, \tau + b}{c \, \tau + d} \right) =
	\varepsilon(d) \, \left( \frac {c \, \overline \tau + d}{c \, \tau + d}
	\right)^{k/2} \, f(\tau)
    \end{equation}
    \item If in addition we have
    \begin{equation} \label{cusp}
	\int_0^1 f(\tau + x) \, dx = 0 \qquad \text{for all $\tau \in \{ x + i
	y \, | \, y>0 \}$;}
    \end{equation}
    \noindent we say that $f$ is a {\em cusp} form.  Otherwise, we
    call $f$ an {\em Eisenstein series}.
\end{enumerate}

For a given irreducible complex representation $\rho$ which is ramified
outside
of a finite number of primes $\Sigma$, the series in \eqref{maass}
always satisfies the first two conditions.  The condition in
\eqref{cusp} is equivalent to $a_{\rho}(0) = 0$, which happens if and
only if $\rho$ is not the trivial representation.  Hence, in order to
invoke Theorem \ref{langlands} it suffices to prove the transformation
property.  Langlands succeeded in proving this in many cases by
proving a generalization of the Selberg Trace Formula.  Unfortunately,
there does not appear to be a way to use these ideas in the
icosahedral case.


\section{Constructing Examples of Icosahedral Representations}

Not many examples satisfying Artin's Conjecture in the icosahedral
case are known.  Those that are may be generated by the following
program, initiated by Joe Buhler \cite{MR58:22019} and furthered by Ian Kiming
\cite{MR95i:11001}.

\begin{enumerate}
    \item Construct an $A_5$-extension $K/\mathbb Q$ by considering quintics.
    \item Consider the discriminant of $K$ in order find possible
    conductors.  This will uniquely specify the representation.
    \item Construct weight 2 cusp forms by considering dihedral
representations.
    \item Divide by an Eisenstein series to find a weight 1 form.
\end{enumerate}

While this program is straightforward, there are two computational
barriers.  First, finding the discriminant of an extension can be a tedious
procedure.  Unfortunately, the method above relies on this step in order to
pinpoint the representation and to generate the weight 2 cusp forms.
Second, division by Eisenstein series is much more difficult that it
sounds.  One must find the zeroes of the weight 2 cusp forms, the
zeroes of the Eisenstein series, and then show that they occur at the
same places.

Motivated by these difficulties, we ask

{\em \begin{center}
    Can we generate Buhler's/Kiming's examples by using elliptic curves?
\end{center}}

\noindent We modify the program above with this question in mind.

\begin{enumerate}
    \item Construct an $A_5$-extension $K/\mathbb Q$ by considering
    quintics.  Use classical results due to Klein to associate elliptic curves.
    \item Construct the $A_5$-representations by considering the 5-torsion.
    \item Construct weight 2 cusp forms associated to the elliptic curve.
    \item Multiply by an Eisenstein series to find a weight 5 form.
\end{enumerate}

\subsection{Step \#1: Relating $A_5$-Extensions to Elliptic Curves}

Felix Klein showed how to associate an elliptic curve to a certain
class of polynomials.

\begin{klein}[Klein \cite{MR96g:01046}]
Fix $q(x) = x^5 + A \, x^2 + B \, x + C$ as a polynomial with rational
coefficients, and assume that $q(x)$ has Galois group $A_5$.  Once
one solves for $j$ in the system

\begin{equation} \begin{aligned}
    A & = - \frac {20}{j} \, \left[ 2 \, m^3 + 3 \, m^2 \, n
    + 432 \, \frac {6 \, m \, n^2 + n^3}{1728 - j} \right] \\
    B & = - \frac 5j \, \left[ m^4 -  864 \, \frac {3 \, m^2 \,
    n^2 + 2 \, m \, n^3}{1728 - j} +  559872 \, \frac {n^4}{(1728 - j)^2}
\right] \\
    C & = - \frac 1j \, \left[ m^5 - 1440 \, \frac {m^3 \, n^2}{1728 -
    j} + 62208 \, \frac {15 \, m \, n^4 + 4 \, n^5}{(1728 - j)^2} \right] \\
\end{aligned} \end{equation}

\noindent then every root of $q(x)$ can be expressed in terms of the
5-torsion on any elliptic curve $E$ with $j = j(E) \in \mathbb
Q(\sqrt{5})$. \end{klein}

As an example, consider the quintic $x^5 + 10 \, x^3 - 10 \,
x^2 + 35 \, x - 18$.  This is not in
the form of the principal quintic above, but after making the substitution

\begin{equation}
    x \mapsto \frac {(1+\sqrt{5}) \, x + (10-30 \sqrt{5})}{2 \, x + (35 +
5\sqrt{5})}
\end{equation}

\noindent the polynomial of interest is

\begin{equation}
    x^5 - 125 \left( 185 + 39 \sqrt{5} \right) \, x^2 - 6875 \left( 56 +
    19 \sqrt{5} \right) \, x - 625 \left( 10691 + 2225 \sqrt{5} \right)
\end{equation}

\noindent with ramified primes $\Sigma = \{ 2, \, \sqrt{5} \}$.  One
solves the equations above to find the elliptic curve

\begin{equation}
    E_0 : \quad y^2 = x^3 + (5 - \sqrt{5}) \, x^2 + \sqrt{5} \, x;
    \qquad j(E_0) = 86048 - 38496 \, \sqrt{5}.
\end{equation}

\subsection{Step \#2: Constructing Icosahedral Representations}

Once the ellipticcurve is found, one constructs the icosahedral representation by following a
rather simple algorithm.

\begin{l-series}[Goins \cite{G1}, Klute \cite{MR98e:11067}]
Let $E$ be an elliptic curve as constructed above, and $\Sigma$ a set
containing ramified primes.  Then there exists an icosahedral representation
$\rho_E$ with L-series

\begin{equation}
    L_{\Sigma} \left( \rho_E, \, s \right) = \prod_{\mathfrak p \notin
    \Sigma} \left( 1- \frac {a_E(\mathfrak p)}{\mathbb N \, \mathfrak
    p^s} + \frac {\omega_5(\mathbb N \, \mathfrak p)}{\mathbb N \, \mathfrak
    p^{2s}} \right)^{-1}
\end{equation}

\noindent where $\omega_5$, a Dirichlet character modulo 5, is the
nebentype; and $a_E(\mathfrak p) \in \mathbb Q(i, \sqrt{5})$ is the
trace of Frobenius.  \end{l-series}

We explain how this works in the simplest case, when the elliptic
curve is defined over $\mathbb Q(\sqrt{5})$.  Given a prime number
$p$, denote the prime ideal lying above $p$ as

\begin{equation}
    \mathfrak p = \left \{ \left. a + \frac {-1+\sqrt{5}}2 \, b \in \mathbb
    Q(\sqrt{5}) \right| a, \, b \in \mathbb Z; \text{ $p$ divides $a^2 - a
\, b -
    b^2$} \right \}
\end{equation}

\noindent The nebentype may be expressed as

\begin{equation}
    \omega_5 ( \mathbb N \, \mathfrak p ) = \left( \frac 5p
    \right) \qquad \text{where} \qquad
    \mathbb N \, \mathfrak p = \begin{cases} p & \text{if $p \equiv
    1, 4 \mod 5$;} \\ p^2 & \text{if $p \equiv 2, 3 \mod 5$.} \end{cases}
\end{equation}

\noindent To calculate the trace of Frobenius, one would factor the polynomial

\begin{equation}
    (x+3)^3 \, (x^2 + 11 \, x + 64) - j(E) \mod \mathfrak p
\end{equation}

\noindent consult the table

\begin{center} \begin{tabular}{c|cccc}
    Irred. Factors & Linear & Quadratics & Cubic & Quintic \\ \hline
    $\dfrac {a_E(\mathfrak p)^2}{\omega_5(\mathbb N \, \mathfrak p)}$
    & 4 & 0 & 1 & $ \left( \dfrac {-1 \pm \sqrt{5}}2 \right)^2$ \\
\end{tabular} \end{center}

\noindent and finally decide upon which square root by the congruence

\begin{equation}
    a_E(\mathfrak p) \equiv  \mathbb N \, \mathfrak p + 1 - \mathbb
    \vert \widetilde E (\mathbb F_{\mathfrak p}) \vert \mod{\left(2-i,
    \, \sqrt{5} \right)}
\end{equation}

\noindent where $\vert \widetilde E (\mathbb F_{\mathfrak p}) \vert$
is the number of points on the elliptic curve mod $\mathfrak p$.

As an application, we take a closer look at the elliptic curve associated
to the
polynomial $x^5 + 10 \, x^3 - 10 \, x^2 + 35 \, x - 18$.

\begin{q-curve} \label{q-curve}
Let $E_0$ be the elliptic curve $y^2 = x^3 + (5 - \sqrt{5}) \, x^2 +
\sqrt{5} \, x$.

\begin{enumerate}
    \item $E_0$ is isogeneous over $\mathbb Q(\sqrt{5}, \, \sqrt{-2})$
    to each of its Galois conjugates.  That is, $E_0$ is a $\mathbb
    Q$-curve.
    \item There is a character $\chi_0$ such that $\chi \otimes \rho_{E_0}$ is
    the base change of an icosahedral representation $\rho$ with conductor
    $N(\rho) = 800$ and nebentype $\epsilon_{\rho} = \left( \frac{-1}{\cdot}
    \right)$.  Specifically, $\rho$ is the icosahedral Galois representation
    studied in \cite{MR58:22019}.
\end{enumerate}
\end{q-curve}

\begin{proof}
We sketch the ideas.  The L-series of the twisted representation
$\chi \otimes \rho_{E_0}$ is defined as

\begin{equation}
    L_{\Sigma} \left( \rho_{E_0}, \, \chi, \, s \right) = \prod_{\mathfrak
p \notin
    \Sigma} \left( 1- \frac {\chi(\mathfrak p) \, a_{E_0}(\mathfrak
    p)}{\mathbb N \, \mathfrak p^s} + \frac {\chi(\mathfrak p)^2 \,
    \omega_5(\mathbb N \, \mathfrak p)}{\mathbb N \, \mathfrak
    p^{2s}} \right)^{-1}
\end{equation}

\noindent If we were to find a character $\chi$ such that 1) it is unramified
outside of $\Sigma$; 2) $\chi(\sigma \, \mathfrak p) \, a_{E_0}(\sigma
\, \mathfrak p) = \chi(\mathfrak p) \, a_{E_0}(\mathfrak p)$ for all $\sigma
\in G_{\mathbb Q}$; and 3) $\chi(\mathfrak p)^2 \, \omega_5(\mathbb N \,
\mathfrak p) = \epsilon_{\rho}(\mathbb N \, \mathfrak p)$; then the
L-series would be in the form

\begin{equation}
    L_{\Sigma} \left( \rho_{E_0}, \, \chi, \, s \right) = \prod_{\mathfrak
p \notin
    \Sigma} \left( 1- \frac {\alpha(\mathbb N \, \mathfrak p)}{\mathbb
    N \, \mathfrak p^s} + \frac {\epsilon_{\rho}(\mathbb N \, \mathfrak
    p)}{\mathbb N \, \mathfrak p^{2s}} \right)^{-1} = L_{\Sigma}
    \left( \rho |_{\mathbb Q(\sqrt(5)}, \, s \right)
\end{equation}

\noindent for some representation $\rho$ defined over $\mathbb Q$.
Cleary $\rho$ has nebentype $\epsilon_{\rho}$ and is
unramified outside of $\Sigma$, so that $\rho$ is the unique
representation studied in \cite{MR58:22019}.  It suffices to construct
the character $\chi$.

The elliptic curve $E_0$ is isogeneous to its congujates, which means

\begin{equation}
    a_{E_0}(\sigma \, \mathfrak p) = \left( \frac {-2}{\mathbb N \, \mathfrak
    p} \right) \, a_{E_0}(\mathfrak p) \implies \begin{aligned}
    \chi(\sigma \, \mathfrak p) & = \left( \frac {-2}{\mathbb N \,
    \mathfrak p} \right) \, \chi(\mathfrak p) \\ \chi(\mathfrak
    p)^2  & = \omega_5(\mathbb N \, \mathfrak p)^{-1} \, \left( \frac
    {-1}{\mathbb N \, \mathfrak p} \right) \end{aligned}
\end{equation}

\noindent One constructs the character explicitly by considering the
ideal of $\mathbb Q(\sqrt{5})$ lying above 40.  \end{proof}

We have shown the existence of the icosahedral representation in
\cite{MR58:22019} without the worry of computing the discriminant of
the splitting field.  Moreover, using the algorithm above the
coefficients can be calculated explicitly.

\subsection{Step \#3: Constructing Weight 2 Cusp Forms}

We will exploit the fact that $E_0$ is isogenous to its Galois
conjugates.  While it is not necessary
in general that $E$ be a $\mathbb Q$-curve in
order to find an icosahedral representation using the steps outlined
in the previous subsection, we do need this fact in this specific
case to work with cusp forms.  Indeed, using the formulas in the
previous subsection one can show that there is always a character such
that the twisted icosahedral representation comes from $\mathbb Q$,
but a priori there seems to be little evidence that the elliptic curve
will always be isogenous to its conjugates.

\begin{modularity} \label{modularity}
Let $E_0$ be the elliptic curve $y^2 = x^3 + (5 - \sqrt{5}) \, x^2 +
\sqrt{5} \, x$ and $\Sigma = \{ 2, \, \sqrt{5} \}$.  There is a cusp
form $f_0$ of weight 2, level 160 such that

\begin{equation}
    L_{\Sigma}(\rho, \, s) \equiv L_{\Sigma}(f_0, \, s) \mod{\left(2-i,
    \, \sqrt{5} \right)}
\end{equation}
\end{modularity}

\begin{proof}  We compute the discriminant of the elliptic curve to
see that it has good reduction outside of $\Sigma$.  It is
straightforward to use the Modular Symbol Algorithm \cite{MR99e:11068} to
calculate coefficients and match a cusp form ${f_0}'$ over $\mathbb
Q(\sqrt{5})$ which is also unramified outside $\Sigma$.  By \ref{q-curve},
the twist $\chi \otimes {f_0}'$ is Galois invariant so it must be the
base change of a cusp form $f_0$ over $\mathbb Q$.  We use the
Modular Symbol Algorithm again to find that the level is 160.  \end{proof}

\subsection{Step \#4: Constructing Weight 5 Cusp Forms}

Using an idea of Ken Ribet \cite{MR95d:11056}, we may strip 5 altogether
from the level
as long as we increase the weight.

\begin{deformations}
Let $E_0$ and $\rho$ be as in \ref{modularity}.  There is a cusp form $f_1$ of weight
5, level 32 and nebentype $\epsilon_{\rho} = \left( \frac {-1}* \right)$
such that

\begin{equation}
    L(\rho_0, \, s) \equiv L(f_0, \, s) \equiv L(f_1, \, s)
    \mod{\left(2-i, \, \sqrt{5} \right)}.
\end{equation}
\end{deformations}

\begin{proof}
As defined in \cite{MR53:7949a}, consider the $\ell$-adic Eisenstein series

\begin{equation}
    \mathcal E(\tau) = 1 + \sum_{n=1}^{\infty} a_n(\mathcal E) \, q^n;
    \quad a_n (\mathcal E) = 2 \, \frac {\sum_{d \mid n}
    \epsilon_{\ell}(d)^{-k} \, d^{k-1}}{L \left({\epsilon_{\ell}}^{-k},
    \, 1-k \right)} \in \ell \, \mathbb Z_{\ell}.
\end{equation}

\noindent where $\epsilon_{\ell}$ is the cyclotomic character.
By \cite[Lemme 10]{MR53:7949a}, this has weight $k$, level $\ell$,
nebentype $\epsilon_{\ell}^{-w}$, and satisfies $\mathcal E \equiv 1
\pod{\ell}$.  Setting $w=3$ and $\ell=5$, the product $f_0 \cdot \mathcal E$
has weight 5, level 160, and nebentypus $\epsilon_{\rho}$ so by \cite[Lemme
6.11]{MR52:284},
there is a bona fide eigenform $f_1 \equiv f_0 \cdot \mathcal E$.
Using the Modular Symbol Algorithm one more time we see that $f_1$ has
level 32.  \end{proof}

\bibliographystyle{plain}

\end{document}